\documentclass{amsart}
\usepackage[utf8]{inputenc}
\usepackage[table]{xcolor}
\usepackage{mathtools,tikz-cd,amsmath,amssymb}
\usepackage{enumerate}
\usepackage[parfill]{parskip}
\usepackage{pgf}
\usepackage{pdfpages}
\usepackage{url}

\usepackage{caption}
\usepackage{subcaption}

\usepackage{tikz,tikz-3dplot}

\usepackage{amsthm,amssymb,latexsym,epic,bbm,comment,mathrsfs}
\usepackage[all,2cell]{xy}
\xyoption{2cell}

\allowdisplaybreaks

\newtheorem{theorem}{Theorem}

\newtheorem{prop}[theorem]{Proposition}

\newtheorem{thmx}{Theorem}
\theoremstyle{definition}
\newtheorem{defn}[theorem]{Definition}
\newtheorem*{example}{Example}

\newtheorem{rem}[theorem]{Remark}
\newtheorem{remark}[theorem]{Remark}

\newcommand{\inj}{\hookrightarrow}

\newcommand{\inv}{^{-1}}

\newcommand{\Ext}{\operatorname{Ext}}

\newcommand{\soc}{\operatorname{soc}}

\newcommand{\cO}{{\mathcal O}}
\newcommand{\RepG}{\operatorname{Rep} G}
\newcommand{\RepU}{\operatorname{Rep} U_q}

\newcommand{\jc}{\mathcal J}

\newcommand{\JM}{\mathbf{JM}}

\newcommand{\JI}{\mathbf{JI}_n}

\newcommand{\cat}{\cO_0}

\newcommand{\tetra}{\mathbf{BSub}_n}
\newcommand{\tetracoord}{\mathbf{Tetr}_n}

\newcommand{\so}{\operatorname{soc}\Delta_{e}/\Delta}
\newcommand{\shadows}{\mathbf{Sub}_n^\jc}
\newcommand{\sub}{\mathbf{Sub}_n}
\newcommand{\domver}{\Delta_{e}}
\newcommand{\ideal}{\mathbf{PI}}
\newcommand{\Ver}{\mathbf{Ver}_n}

\newcommand{\al}{\mathbf{ASM}}
\newcommand{\sbf}{\mathbf{S}}

\newcommand{\s}{s_1}
\renewcommand{\t}{s_2}
\newcommand{\cor}{=}

\newcommand{\hk}[1]{}

\begin{document}

\title[Alternating sign matrices and Verma modules]
{Alternating sign matrices and Verma modules}

\author[Hankyung Ko]{Hankyung Ko}

\begin{abstract}
We show that the poset of alternating sign matrices, with Bruhat order, is isomorphic to the poset of certain submodules of the dominant Verma module for the special linear Lie algebra $\mathfrak{sl}_n$. 
The latter poset consists of the intersections of Verma submodules and can also be defined in terms of a Kazhdan-Lusztig cell.
\end{abstract}

\maketitle

\section*{}

Let $\sbf_n$ be the set of permutations on $n$ letters. 
With the natural group structure given by the composition of permutations, the group $\sbf_n$ is part of the Coxeter system $(W,S)$, with the group $W=\sbf_n$ and the generating set
\[S = \{s_1,\cdots, s_{n-1}\},\]
where $s_i$ is the transposition swapping the letters $i$ and $i+1$.
In particular, each element $w \in W$ is written as $w=st\cdots u$ with $s,t,\cdots,u\in S$. Such expressions of the minimal length are called \emph{reduced expressions} of $w$.  
The Bruhat order `$\leq$' on $\sbf_n$ is equivalent to the Coxeter order on $(W,S)$, namely, for $x,y\in W$ we have $x\leq y$ if and only if $x$ has a reduced expression that is a subexpression of a reduced expression of $y$.
A left (resp., right) \emph{descent} of $w\in W$ is defined to be the elements $s\in S$ such that $sw<w$ (resp., $ws<w$); a left (resp., right) \emph{ascent} of $w\in W$ is $s\in S$ such that $sw>w$ (resp., $ws>w$).
The poset $\sbf_n$ has a unique minimal element $e\in W$, the identity element in $W$, and a unique maximal element $w_0\in W$.

If we view the elements in $\sbf_n$ as $n\times n$ permutation matrices, the following partial order is equivalent to the Bruhat order on $\sbf_n$ (see \cite[Theorem 2.1.5]{BjBr}).
For $A,B\in \sbf_n$, we have $A\leq B$ if and only if (writing $A=(A_{ij})$ and $B=(B_{ij})$)
\[ \sum_{i\leq k, j\leq l}A_{ij}\geq \sum_{i\leq k, j\leq l}B_{ij}\]
for all $1\leq k,l\leq n$.

An $n\times n$ \emph{alternating sign matrix} is a matrix $A=(A_{ij})$ such that
\begin{itemize}
    \item for each $1\leq i,j\leq n$, we have $A_{ij}\in\{0,1,-1\}$;
    \item each row (resp., column) sums to $1$, i.e., $\sum_{k}A_{ik}=1$ and $\sum_k A_{kj}=1$ for each $1\leq i,j\leq n$;
    \item the nonzero entries in each row (resp., column) alternate the sign. 
\end{itemize}
See \cite{MRRalt,manyfacesAl}.
A permutation matrix can be viewed as an alternating sign matrix whose entries are nonnegative.
Let $\al_n$ be the set of $n\times n$ alternating sign matrices.
The matrix definition of the Bruhat order above extends naturally to $\al_n$, namely, we set for $A,B\in \al_n$
\[A\leq B \Longleftrightarrow \sum_{i\leq k, j\leq l}A_{ij}\geq \sum_{i\leq k, j\leq l}B_{ij} \text{ for all $1\leq k,l\leq n$.}\]

A result of Lascoux-Sch{\"u}tzenberger \cite{LS} (see also \cite{AlBruhat} for illustrations) says that the embedding of posets 
\begin{equation}\label{S in Al}
  \sbf_n\inj \al_n  
\end{equation}
realizes the MacNeille completion of the Bruhat order on $\sbf_n$. That is, the poset $\al_n$ is the smallest lattice that contains $\sbf_n$.

The poset $\sbf_n$ has categorical lifts in Lie theoretic contexts. 
We consider the one involving Verma modules of the special linear Lie algebra $\mathfrak{g}=\mathfrak{sl}_n$ over $\mathbb C$.
Let us fix Cartan and Borel subalgebras
\begin{equation}\label{nottriangulardec}
    \mathfrak{h}\subset \mathfrak{b}\subset\mathfrak{g}.
\end{equation}
Associated to \eqref{nottriangulardec} is the root system 
and the Weyl group $(W,S)$ which acts on $\mathfrak{h}^*$. 
The Weyl group $(W,S)$ agrees with the Coxeter group $(W,S)$ above. The set $S$ consists of the simple reflections, i.e., reflections with respect to simple roots.

The Verma module of highest weight $\lambda\in \mathfrak{h}^*$ is defined by inflating the 1-dimensional $\mathfrak{h}$-module $\lambda$ to a $\mathfrak{b}$-module $\mathbb C_\lambda$ 
and then inducing to $\mathfrak{g}$. 
That is, we let
\begin{equation}\label{eqver}
\Delta(\lambda ) = U(\mathfrak{g})\otimes_{U(\mathfrak{b})} \mathbb C_\lambda.    
\end{equation}
Each Verma module $\Delta(\lambda)$ has the head isomorphic to the simple $\mathfrak{g}$-module $L(\lambda)$ of highest weight $\lambda$.  
We restrict ourselves to the Verma modules whose highest weights are of the form $w.0 = w\rho - \rho$, where $2\rho$ is the sum of all positive roots. 
(These are exactly the Verma modules and the simple modules in 
the principal block of the BGG category $\cO$; see Section \ref{s:ver}.)
Note that $w.0\neq x.0$ for $w\neq x\in W$.
Our shorthand for such Verma and simple modules is
\[\Delta_w := \Delta(w.0)\quad L_w :=L(w.0).\]

Each Verma module $\Delta_w$, for $w\in W$, is canonically a submodule of $\domver$ via the unique (up to scalar) map $\Delta_w\inj \domver$. We thus identify each $\Delta_w$ with the corresponding submodule in $\domver$. 
Then we have the poset isomorphism
\begin{equation}\label{posetiso}
    \sbf_n \cong (\{\Delta_w\subseteq \Delta_e\ |\ w\in W\},\supseteq)=:\Ver,
\end{equation}
 that is, we have for all $w,x\in W$
\[\Delta_w\subseteq \Delta_x \Longleftrightarrow w\geq x.\]

Our goal, achieved in Theorem~\ref{theoremattheend}, is to give a representation theoretic interpretation of the MacNeille completion
$\sbf_n\subseteq \al_n$ extending the isomorphism \eqref{posetiso}:

\begin{thmx}\label{mainthm}
There is a poset isomorphism
\[\al_n\cong (\{\bigcap_{x\in U}\Delta_x\subseteq \Delta_e\ |\ U\subseteq W\},\supseteq)=:\overline{\Ver}\]
such that the diagram of posets 
\[ \xymatrix{
 \al_n \ar[r]^-\simeq & \overline{\Ver} \\
 \sbf_n \ar[r]^-\simeq
 \ar@{}[u]|{\rotatebox[origin=c]{90}{$\subseteq$}} & \Ver
 \ar@{}[u]|{\rotatebox[origin=c]{90}{$\subseteq$}}   
} \]
is commutative.
\end{thmx}

The intersections $\bigcap_{x\in U}\Delta_x$ have a nice characterization implicit in \cite{kmmBig}.
The characterization is in terms of the Kazhdan-Lusztig two-sided (pre)order $\leq_J$ on $\sbf_n$ with respect to the Kazhdan-Lusztig basis \cite{KL79} (see \cite[Section 2.3 and Section 1.1]{kmmBig} for definitions and more information; see also \cite{geckcellA} which describes the order $\leq_J$ via the Robinson-Schensted correspondence as the dominance order on partitions).
With respect to $\leq_J$, the element $w_0\in W$ is maximal and forms, by itself, the ultimate two-sided cell $\{w_0\}$. 
All $w_0s$ (and $sw_0$) with $s\in S$ belong to the same two-sided cell $\jc$, called the \emph{penultimate cell}, which is maximal in $W\setminus \{w_0\}$ with respect to $\leq_{J}$.
The set $\jc$ is thought of as an $(n-1)\times (n-1)$ grid: it contains for each $1\leq i,j\leq n-1$ exactly one element $z\in W$ with left (resp., right) ascent $\{s_i\}$ (resp., $\{s_j\}$) and no other elements.

Then we have (see Proposition~\ref{shadows=intver})
\begin{equation}\label{introtwo}
  \{\bigcap_{x\in U}\Delta_x\subseteq \Delta_e\ |\ U\subseteq W\}   = \{M\subseteq \domver\text{ gradable} \ |\ [\soc(\domver / M): L_z] =0 \text{ for $z\not\in \jc$}\}=:\shadows.  
\end{equation}
The gradable condition in the right hand side refers to the positive grading explained in Section~\ref{s:ver}, although the condition is equivalent to requiring that the socle series of $M$ agrees with the socle series of $\Delta_e$ intersected with $M$, which needs no reference to the grading.
 
\begin{example}
Let $n=3$. We label the simple reflections as
$S=\{\s, \t\}$. 
The two-sided order on $\sbf_3$ is $\{e\}<\jc=\{\s,\t,\s\t,\t\s\}<\{w_0\}$, where $w_0=\s\t\s=\t\s\t$, and
 $\sbf_3$ is embedded into the poset $\al_3$ as displayed in Figure~\ref{al_3}. The only element in $\al_3$ that does not belong to $\sbf_3$ can be written as $\s\vee\t$.
The corresponding submodules of $\domver$ and their visualization are in Figure~\ref{Ver_3}. A circle at $(w,d)$ corresponds to the composition factor $L_w$ (with $w\in \jc$) of $\Delta_e$ in the $d$-th radical layer; the circle being filled (resp., empty) encodes the composition factor being (resp., not being) contained in $\Delta_w\subseteq \Delta_e$.  
Note for any $U\subset \sbf_3$ that the intersection $\bigcap_{x\in U}\Delta_x $ coincides with one of the displayed submodules, e.g., $\Delta_{\s\t}\cap\Delta_{\t\s} = \Delta_{w_0}$.

\begin{figure}
\[\xymatrix@=2mm{
&*+{e= \left[\begin{array}{ccc} 1& 0  &0   \\0  &1   &0   \\ 0 &0   &1  \end{array}\right]}
\ar@{-}[dl]\ar@{-}[dr]
&\\
*+{\s=\left[\begin{array}{ccc} 0& 1  &0  \\ 1 & 0  &0   \\0  &0   &1  \end{array}\right]}
\ar@{-}[dr]
&&*+{\t=\left[\begin{array}{ccc}  1 & 0  &0   \\ 0 &0  &1   \\0  &1   &0  \end{array}\right]}
\ar@{-}[dl]
\\
&*+{\left[\begin{array}{ccc} 0 &1   &0   \\1  &-1   &1   \\0 &1   &0  \end{array}\right]}
\ar@{-}[dl]\ar@{-}[dr]
\\
*+{\t\s=\left[\begin{array}{ccc}0 &1   &0   \\0  &0   &1   \\ 1 & 0  &0  \end{array}\right]}
\ar@{-}[dr]
&&*+{\s\t=\left[\begin{array}{ccc}0 &0   &1   \\ 1 &0  &0   \\0  & 1  &0  \end{array}\right]}\ar@{-}[dl]
\\
&*+{w_0=\left[\begin{array}{ccc} 0  &0   &1   \\ 0 & 1  & 0  \\1  & 0  &0  \end{array}\right]}
&
}
\]
\caption{The poset $\al_3$}
\label{al_3}
\end{figure}

\begin{figure}[ht]
    \centering
$\Delta_e \cor$ 
\tdplotsetmaincoords{110}{130} 
\begin{tikzpicture}[baseline=(AA.south),tdplot_main_coords, scale=1.3]

\filldraw[black] (1, 1, -1) circle (1.5pt) node[anchor=west](AA) {\tiny $(\t, 1)$};
\filldraw[black] (1, 2, -2) circle (1.5pt) node[anchor=north] {\tiny $(\t\s, 2)$};
\filldraw[black] (2, 1, -2) circle (1.5pt) node[anchor=north] {\tiny $(\s\t, 2)$};
\filldraw[black] (2, 2, -1) circle (1.5pt) node[anchor=west] {\tiny $(\s, 1)$};

\draw (1, 1, -1) -- (1, 2, -2);
\draw (1, 1, -1) -- (2, 1, -2);
\draw (2, 2, -1) -- (1, 2, -2);
\draw (2, 2, -1) -- (2, 1, -2);

\draw[gray, dashed] (1, 1, -2) -- (1, 2, -2) -- (2, 2, -2) -- (2, 1, -2) -- cycle;
\draw[gray, dashed] (1, 1, -1) -- (1, 2, -1) -- (2, 2, -1) -- (2, 1, -1) -- cycle;
\draw[gray, dashed] (1, 1, -2) -- (1, 1, -1);
\draw[gray, dashed] (1, 2, -2) -- (1, 2, -1);
\draw[gray, dashed] (2, 2, -2) -- (2, 2, -1);
\draw[gray, dashed] (2, 1, -2) -- (2, 1, -1);
\draw[gray, dashed] (1, 2, -2) -- (2, 1, -2);
\draw[gray, dashed] (1, 1, -1) -- (2, 2, -1);

\end{tikzpicture}

$\Delta_{\s} \cor$ 
\tdplotsetmaincoords{110}{130} 
\begin{tikzpicture}[baseline=(AA.south),tdplot_main_coords, scale=1.3]

\draw[black] (1, 1, -1) circle (1.5pt) node[anchor=west](AA) {\tiny $(\t, 1)$};
\filldraw[black] (1, 2, -2) circle (1.5pt) node[anchor=north] {\tiny $(\t\s, 2)$};
\filldraw[black] (2, 1, -2) circle (1.5pt) node[anchor=north] {\tiny $(\s\t, 2)$};
\filldraw[black] (2, 2, -1) circle (1.5pt) node[anchor=west] {\tiny $(\s, 1)$};

\draw[gray, dashed] (1, 1, -1) -- (1, 2, -2);
\draw[gray, dashed] (1, 1, -1) -- (2, 1, -2);
\draw (2, 2, -1) -- (1, 2, -2);
\draw (2, 2, -1) -- (2, 1, -2);

\draw[gray, dashed] (1, 1, -2) -- (1, 2, -2) -- (2, 2, -2) -- (2, 1, -2) -- cycle;
\draw[gray, dashed] (1, 1, -1) -- (1, 2, -1) -- (2, 2, -1) -- (2, 1, -1) -- cycle;
\draw[gray, dashed] (1, 1, -2) -- (1, 1, -1);
\draw[gray, dashed] (1, 2, -2) -- (1, 2, -1);
\draw[gray, dashed] (2, 2, -2) -- (2, 2, -1);
\draw[gray, dashed] (2, 1, -2) -- (2, 1, -1);
\draw[gray, dashed] (1, 2, -2) -- (2, 1, -2);
\draw[gray, dashed] (1, 1, -1) -- (2, 2, -1);

\end{tikzpicture}
$\quad \quad \quad \quad \Delta_{\t} \cor$ 
\tdplotsetmaincoords{110}{130} 
\begin{tikzpicture}[baseline=(AA.south),tdplot_main_coords, scale=1.3]

\filldraw[black] (1, 1, -1) circle (1.5pt) node[anchor=west](AA) {\tiny $(\t, 1)$};
\filldraw[black] (1, 2, -2) circle (1.5pt) node[anchor=north] {\tiny $(\t\s, 2)$};
\filldraw[black] (2, 1, -2) circle (1.5pt) node[anchor=north] {\tiny $(\s\t, 2)$};
\draw[black] (2, 2, -1) circle (1.5pt) node[anchor=west] {\tiny $(\s, 1)$};

\draw (1, 1, -1) -- (1, 2, -2);
\draw (1, 1, -1) -- (2, 1, -2);
\draw[gray, dashed] (2, 2, -1) -- (1, 2, -2);
\draw[gray, dashed] (2, 2, -1) -- (2, 1, -2);

\draw[gray, dashed] (1, 1, -2) -- (1, 2, -2) -- (2, 2, -2) -- (2, 1, -2) -- cycle;
\draw[gray, dashed] (1, 1, -1) -- (1, 2, -1) -- (2, 2, -1) -- (2, 1, -1) -- cycle;
\draw[gray, dashed] (1, 1, -2) -- (1, 1, -1);
\draw[gray, dashed] (1, 2, -2) -- (1, 2, -1);
\draw[gray, dashed] (2, 2, -2) -- (2, 2, -1);
\draw[gray, dashed] (2, 1, -2) -- (2, 1, -1);
\draw[gray, dashed] (1, 2, -2) -- (2, 1, -2);
\draw[gray, dashed] (1, 1, -1) -- (2, 2, -1);

\end{tikzpicture}

$\Delta_{\s}\cap\Delta_{\t} \cor$ 
\tdplotsetmaincoords{110}{130} 
\begin{tikzpicture}[baseline=(AA.south),tdplot_main_coords, scale=1.3]

\draw[black] (1, 1, -1) circle (1.5pt) node[anchor=west](AA) {\tiny $(\t, 1)$};
\filldraw[black] (1, 2, -2) circle (1.5pt) node[anchor=north] {\tiny $(\t\s, 2)$};
\filldraw[black] (2, 1, -2) circle (1.5pt) node[anchor=north] {\tiny $(\s\t, 2)$};
\draw[black] (2, 2, -1) circle (1.5pt) node[anchor=west] {\tiny $(\s, 1)$};

\draw[gray, dashed] (1, 1, -1) -- (1, 2, -2);
\draw[gray, dashed] (1, 1, -1) -- (2, 1, -2);
\draw[gray, dashed] (2, 2, -1) -- (1, 2, -2);
\draw[gray, dashed] (2, 2, -1) -- (2, 1, -2);

\draw[gray, dashed] (1, 1, -2) -- (1, 2, -2) -- (2, 2, -2) -- (2, 1, -2) -- cycle;
\draw[gray, dashed] (1, 1, -1) -- (1, 2, -1) -- (2, 2, -1) -- (2, 1, -1) -- cycle;
\draw[gray, dashed] (1, 1, -2) -- (1, 1, -1);
\draw[gray, dashed] (1, 2, -2) -- (1, 2, -1);
\draw[gray, dashed] (2, 2, -2) -- (2, 2, -1);
\draw[gray, dashed] (2, 1, -2) -- (2, 1, -1);
\draw[gray, dashed] (1, 2, -2) -- (2, 1, -2);
\draw[gray, dashed] (1, 1, -1) -- (2, 2, -1);

\end{tikzpicture}

$\Delta_{\t\s} \cor$ 
\tdplotsetmaincoords{110}{130} 
\begin{tikzpicture}[baseline=(AA.south),tdplot_main_coords, scale=1.3]

\draw[black] (1, 1, -1) circle (1.5pt) node[anchor=west](AA) {\tiny $(\t, 1)$};
\filldraw[black] (1, 2, -2) circle (1.5pt) node[anchor=north] {\tiny $(\t\s, 2)$};
\draw[black] (2, 1, -2) circle (1.5pt) node[anchor=north] {\tiny $(\s\t, 2)$};
\draw[black] (2, 2, -1) circle (1.5pt) node[anchor=west] {\tiny $(\s, 1)$};

\draw[gray, dashed] (1, 1, -1) -- (1, 2, -2);
\draw[gray, dashed] (1, 1, -1) -- (2, 1, -2);
\draw[gray, dashed] (2, 2, -1) -- (1, 2, -2);
\draw[gray, dashed] (2, 2, -1) -- (2, 1, -2);

\draw[gray, dashed] (1, 1, -2) -- (1, 2, -2) -- (2, 2, -2) -- (2, 1, -2) -- cycle;
\draw[gray, dashed] (1, 1, -1) -- (1, 2, -1) -- (2, 2, -1) -- (2, 1, -1) -- cycle;
\draw[gray, dashed] (1, 1, -2) -- (1, 1, -1);
\draw[gray, dashed] (1, 2, -2) -- (1, 2, -1);
\draw[gray, dashed] (2, 2, -2) -- (2, 2, -1);
\draw[gray, dashed] (2, 1, -2) -- (2, 1, -1);
\draw[gray, dashed] (1, 2, -2) -- (2, 1, -2);
\draw[gray, dashed] (1, 1, -1) -- (2, 2, -1);

\end{tikzpicture}
$\quad \quad \quad \quad \Delta_{\s\t} \cor$ 
\tdplotsetmaincoords{110}{130} 
\begin{tikzpicture}[baseline=(AA.south),tdplot_main_coords, scale=1.3]

\draw[black] (1, 1, -1) circle (1.5pt) node[anchor=west](AA) {\tiny $(\t, 1)$};
\draw[black] (1, 2, -2) circle (1.5pt) node[anchor=north] {\tiny $(\t\s, 2)$};
\filldraw[black] (2, 1, -2) circle (1.5pt) node[anchor=north] {\tiny $(\s\t, 2)$};
\draw[black] (2, 2, -1) circle (1.5pt) node[anchor=west] {\tiny $(\s, 1)$};

\draw[gray, dashed] (1, 1, -1) -- (1, 2, -2);
\draw[gray, dashed] (1, 1, -1) -- (2, 1, -2);
\draw[gray, dashed] (2, 2, -1) -- (1, 2, -2);
\draw[gray, dashed] (2, 2, -1) -- (2, 1, -2);

\draw[gray, dashed] (1, 1, -2) -- (1, 2, -2) -- (2, 2, -2) -- (2, 1, -2) -- cycle;
\draw[gray, dashed] (1, 1, -1) -- (1, 2, -1) -- (2, 2, -1) -- (2, 1, -1) -- cycle;
\draw[gray, dashed] (1, 1, -2) -- (1, 1, -1);
\draw[gray, dashed] (1, 2, -2) -- (1, 2, -1);
\draw[gray, dashed] (2, 2, -2) -- (2, 2, -1);
\draw[gray, dashed] (2, 1, -2) -- (2, 1, -1);
\draw[gray, dashed] (1, 2, -2) -- (2, 1, -2);
\draw[gray, dashed] (1, 1, -1) -- (2, 2, -1);

\end{tikzpicture}

$\Delta_{w_0} \cor$ 
\tdplotsetmaincoords{110}{130} 
\begin{tikzpicture}[baseline=(AA.south),tdplot_main_coords, scale=1.3]

\draw[black] (1, 1, -1) circle (1.5pt) node[anchor=west](AA) {\tiny $(\t, 1)$};
\draw[black] (1, 2, -2) circle (1.5pt) node[anchor=north] {\tiny $(\t\s, 2)$};
\draw[black] (2, 1, -2) circle (1.5pt) node[anchor=north] {\tiny $(\s\t, 2)$};
\draw[black] (2, 2, -1) circle (1.5pt) node[anchor=west] {\tiny $(\s, 1)$};

\draw[gray, dashed] (1, 1, -1) -- (1, 2, -2);
\draw[gray, dashed] (1, 1, -1) -- (2, 1, -2);
\draw[gray, dashed] (2, 2, -1) -- (1, 2, -2);
\draw[gray, dashed] (2, 2, -1) -- (2, 1, -2);

\draw[gray, dashed] (1, 1, -2) -- (1, 2, -2) -- (2, 2, -2) -- (2, 1, -2) -- cycle;
\draw[gray, dashed] (1, 1, -1) -- (1, 2, -1) -- (2, 2, -1) -- (2, 1, -1) -- cycle;
\draw[gray, dashed] (1, 1, -2) -- (1, 1, -1);
\draw[gray, dashed] (1, 2, -2) -- (1, 2, -1);
\draw[gray, dashed] (2, 2, -2) -- (2, 2, -1);
\draw[gray, dashed] (2, 1, -2) -- (2, 1, -1);
\draw[gray, dashed] (1, 2, -2) -- (2, 1, -2);
\draw[gray, dashed] (1, 1, -1) -- (2, 2, -1);

\end{tikzpicture}
\caption{The elements in $\overline{\mathbf{Ver}_3}$}
\label{Ver_3}
\end{figure}
\end{example}

\begin{rem}
The isomorphism $\Ver\cong (W,\leq)$ holds for all semisimple complex Lie algebras $\mathfrak{g}$ and their Weyl groups $W$ with the Bruhat order $\leq$. 
However, in all types other than $A$ (and other than the dihedral types), the equality \eqref{introtwo} does not hold. 
This follows from a recent work \cite{kmmJoin} where related homological problems are discussed.
\hk{We have $\subsetneq$. for DEF the multiplicity can be 2 so the rhs is infinite while lhs is always finite. for type B consider the situation with ox in the same level same descent. there's $y\in W$ of type $O_B$ which has ox as soc. suppose $N$ is what has o as the socle. we claim N is not the intersection of vermas. note that x is not in $\Delta_y$ while $N$ has it, so $N$ is not contained in $\Delta_y$. but by Prop \ref{max=soc} and since $y$ is the only $w\in W$ with o in the socle, we need $U$ to contain $y$. contradiction.  }
Also, neither the poset $\overline{\Ver}$ nor the poset $(\shadows,\supseteq)$ seems to agree with the MacNeille completion of the Bruhat order on the Weyl group in general. 
\hk{for the rhs it is clear (in type B MC does not contain the N above). 

the lhs actually seems to agree with the MN in type B (type B is dissective so the MacNeille agrees with the ideal poset. the ideal poset should agree with the intersection of vermas!!). 

for others recall the example where join and intersection don't agree when U is not JM.
in the type E example, we have $x\vee y = b$ while $\Delta_x\cap\Delta_y \neq \Delta_b$. so $\Ver\subset \overline{\Ver}$ is not meet dense and thus is not MacNeille completion (see harding's lecture 5, which says $\overline\Ver$ is not the ideal completion either)
question: KMM question (saying join=intersection for all U) implies $\overline\Ver$ is the ideal completion? should it hold in type B?}
However, it is a good question whether either $\overline{\Ver}$ or $\shadows$ is related to some variation of the alternating sign matrices (see \cite{kuperbergXASMs} for several of them).
\end{rem}

\begin{rem}
A similar realization of the alternating sign matrices might exist for other Lie theoretic categories. 
We do not look into this in the current paper, except we easily derive an analogue of our main theorem for the periplectic Lie superalgebra using the result in \cite{ChM}. 
We explain this example in Remark~\ref{perip}.

\hk{Rep SLn seems difficult but the modular category O needs to be considered!}
\hk{For quantum groups at rou Andersen-Kaneda gives morphisms between vermas preserving the bruhat order. 
they are not injective and there are other morphisms between vermas (maybe wrong graded though) but the socle of cokernel might be determinable and of the form $L_{zt}$ where $z\in \jc$ and $t\in \mathbb Z R$ where we view the affine weylf group $W$ as $W_f$ semidrect $\mathbb Z R$. 
the translation $t$ should be determined by $w$. 
More precisely, if we have $w\leq w'$ in the affine Weyl group, we write them as $w=xt$ and $w'= x't'$.
If $t\neq t'$ there is no morphism (maybe this and everything requires `not too close to the wall' i.e., $X^{++}$ in AK).
If $t=t'$ then $x\leq x' \in W_f$ so we can do the usual finite weyl group combinatorics. 
Perhaps we have
\[\soc coker (\Delta_{w} \to \Delta_{w_0t}) = \oplus_{z\in\JM(ww_0)} L_{\phi(z)t}\]
where $\phi(z)\in\jc$ is so that $\so_z = L_{\phi(z)}$ for the finite case.
}
\end{rem}

\subsection*{Acknowledgments}
We thank Valentin Buciumas for comments.

\section{Verma modules}

\subsection{Category O}\label{s:ver}
Verma modules are best discussed in the Bernstein-Gelfand-Gelfand category $\cO$. 
The category $\cO$ is defined from a triangular decomposition 
\begin{equation}\label{triangulardec}
\mathfrak{g} = \mathfrak{n}^+\oplus\mathfrak{h}\oplus\mathfrak{n}^-.
\end{equation}
The category $\cO$, by definition, consists of the $\mathfrak{g}$-modules
which are finitely generated, semisimple over $\mathfrak{h}$, and locally finite over $\mathfrak{n}^+$.
The decomposition~\eqref{triangulardec} fixes Cartan and Borel subalgebras as in \eqref{nottriangulardec} by letting $\mathfrak{b} = \mathfrak{n}^+\oplus\mathfrak{h}$, and the Verma modules, as defined in \eqref{eqver}, belong to $\cO$.

The simple objects in $\cO$ arise as the head of Verma modules. The head of $\Delta(\lambda)$, for $\lambda\in\mathfrak{h}^*$, is simple and is of highest weight $\lambda$. 
We denote this simple module by $L(\lambda)$.
The set $\{L(\lambda)\}_{\lambda\in\mathfrak{h}^*}$ gives a complete irredundant list of simple objects in $\cO$ up to isomorphism.

The category $\cO$ decomposes into blocks each of which has finitely many simple objects. 
We will work with a largest (regular) block: the \emph{principal block} $\cO_0$ is the block of $\cO$ which contains the trivial $\mathfrak{g}$-module.  
The Verma modules in $\cO_0$ are exactly $\Delta_w = \Delta(w.0)$ for $w\in W$ and the simple objects in $\cO_0$ are exactly $L_w=L(w.0)$ for $w\in W$. 

By Soergel \cite{SoeKatO}, the category $\cO_0$ has a ($\mathbb Z$-)graded lift where the grading is moreover Koszul. 
The Verma modules are gradable and their grading agrees (up to shift) with the radical series. 
By slight abuse, we denote by $L_w$ and $\Delta_w$ the graded simple and Verma modules and denote by $\cO_0$ the graded principal block of $\cO$ for the rest of the paper.
We set $L_w$ to be concentrated in degree $0$ and $\Delta_w$ to have its head in degree $\ell(w)$.
Here $\ell(w)$ is the Coxeter length of $w$ in $(W,S)$, which is by definition the length of a reduced expression of $w$.
We denote by $\langle -\rangle$ the grade shift to the positive direction, so that $L_w\langle d \rangle$ is in degree $d$.
In other words, we have for each $w\in W$ the short exact sequence 
\[0\to \operatorname{rad} \Delta_w\to \Delta_w \to L_w\langle \ell(w)\rangle \to 0\]
of graded morphisms in $\cO_0$.
Our setting also makes, for $w\geq x$, the inclusions  $\Delta_w\inj \Delta_x$
graded.

\subsection{The penultimate cell and join-irreducible elements}\label{s:JandJI}

We now state a few results from \cite{kmmBig} which provide a groundwork for the current paper. 
Recall that $\jc$ denotes the penultimate Kazhdan-Lusztig two-sided cell in $W$.

\begin{prop}\cite[Proposition 3]{kmmBig}\label{socinJ}
For each $w,x\in W$, if (a shift of) $L_x$ appears as a summand of $\so_w$, then we have $x\in \jc$.
\end{prop}

\begin{prop}\cite[Proposition 12]{kmmBig}\label{multfreeJ}
For each $z\in \jc$ and $d\in \mathbb Z$ we have
\[[\domver:L_z \langle d \rangle] \leq 1.\]
\end{prop}

For the following result, we introduce the set $\JI$ of join-irreducible elements of the post $\sbf_n$.
An element in $w\in\sbf_n$ is called \emph{join-irreducible} if $w$ is not the join of some $U\subseteq \sbf_n\setminus\{w\}$.
It is shown in \cite[Section 4]{LS} that $\JI$ are exactly the elements in $W$ with unique left and right descents
. Such an element is called \emph{bigrassmannian}. 
Note that \cite{kmmBig} states everything in in terms of the bigrassmannian elements.

\begin{prop}\cite[Theorem 1]{kmmBig}\label{tetraJIA}
For $w\in W$ we have
\[\so_w \text{ is simple} \Longleftrightarrow w\in \JI .\]
Moreover, the assignment 
\begin{equation}\label{Phi}
\begin{split}
    \Psi: \JI &\to \{\text{the subquotients $L$ in $\domver$ such that $L\cong L_z$ for $z\in \jc$} \}\\
     w&\mapsto \so_w 
\end{split}
\end{equation}
 is a bijection.
\end{prop}

\subsection{Tetrahedral poset}\label{s:tetra}

The bijection \eqref{Phi} identifies $\JI$ with a combinatorial tetrahedron as follows. Let for $1\leq i,j\leq n-1$
\[d(i,j) = \min\{i,j,n-i,n-j\}-1\]
and for $0\leq k \leq d(i,j)$
\[g(i,j,k) = \frac{(n-1)(n-2)}{2}-|i-j|-2k.\]
Then consider the set
\[\tetracoord = \{(i,j,g(i,j,k))\in\mathbb Z^3 \ |\ 1\leq i,j\leq n-1, 0\leq k \leq d(i,j)\} .\] 
The set $\tetracoord$ is constructed so that the assignment from $\JI$ to $\tetracoord$ given by (using Proposition~\ref{tetraJIA})
$w\mapsto (i,j,g)$, where 
\[\so_w\cong L_z\langle g\rangle\] such that $z\in\jc$ is the element with unique left ascent $s_i$ and right ascent $s_j$, is bijective. See \cite[Section 4]{kmmBig} for details.

We make $\tetracoord$ into a poset generated by the relations
\[(i,j,g)\leq(i\pm 1,j,g+1), \quad (i,j,g)\leq(i,j\pm 1,g+1)\]
(whenever both sides belong to $\tetracoord$).
See Figure~\ref{fig_tetra} for an illustration where the edges represent the generating relations.

\begin{figure}[ht]
\centering
\tdplotsetmaincoords{110}{130} 
\begin{tikzpicture}[tdplot_main_coords, scale=1.4]

\filldraw[black] (1, 1, -1) circle (1.5pt) node[anchor=west] {\tiny $(1, 1, 1)$};
\filldraw[black] (1, 2, -2) circle (1.5pt) node[anchor=north] {\tiny $(1, 2, 2)$};
\filldraw[black] (2, 1, -2) circle (1.5pt) node[anchor=north] {\tiny $(2, 1, 2)$};
\filldraw[black] (2, 2, -1) circle (1.5pt) node[anchor=west] {\tiny $(2, 2, 1)$};

\draw (1, 1, -1) -- (1, 2, -2);
\draw (1, 1, -1) -- (2, 1, -2);
\draw (2, 2, -1) -- (1, 2, -2);
\draw (2, 2, -1) -- (2, 1, -2);

\draw[gray, dashed] (1, 1, -2) -- (1, 2, -2) -- (2, 2, -2) -- (2, 1, -2) -- cycle;
\draw[gray, dashed] (1, 1, -1) -- (1, 2, -1) -- (2, 2, -1) -- (2, 1, -1) -- cycle;
\draw[gray, dashed] (1, 1, -2) -- (1, 1, -1);
\draw[gray, dashed] (1, 2, -2) -- (1, 2, -1);
\draw[gray, dashed] (2, 2, -2) -- (2, 2, -1);
\draw[gray, dashed] (2, 1, -2) -- (2, 1, -1);
\draw[gray, dashed] (1, 2, -2) -- (2, 1, -2);
\draw[gray, dashed] (1, 1, -1) -- (2, 2, -1);

\end{tikzpicture}
\tdplotsetmaincoords{110}{130} 
\begin{tikzpicture}[tdplot_main_coords, scale=1.4]

\filldraw[black] (1, 1, -3) circle (1.5pt) node[anchor=west] {\tiny $(1, 1, 3)$};
\filldraw[black] (1, 2, -4) circle (1.5pt) node[anchor=west] {\tiny $(1, 2, 4)$};
\filldraw[black] (1, 3, -5) circle (1.5pt) node[anchor=north] {\tiny $(1, 3, 5)$};
\filldraw[black] (2, 1, -4) circle (1.5pt) node[anchor=east] {\tiny $(2, 1, 4)$};
\filldraw[black] (2, 2, -5) circle (1.5pt) node[anchor=north] {\tiny $(2, 2, 5)$};
\filldraw[black] (2, 2, -3) circle (1.5pt) node[anchor=west] {\tiny $(2, 2, 3)$};
\filldraw[black] (2, 3, -4) circle (1.5pt) node[anchor=west] {\tiny $(2, 3, 4)$};
\filldraw[black] (3, 1, -5) circle (1.5pt) node[anchor=north] {\tiny $(3, 1, 5)$};
\filldraw[black] (3, 2, -4) circle (1.5pt) node[anchor=east] {\tiny $(3, 2, 4)$};
\filldraw[black] (3, 3, -3) circle (1.5pt) node[anchor=west] {\tiny $(3, 3, 3)$};

\draw (1, 1, -3) -- (1, 2, -4);
\draw (1, 1, -3) -- (2, 1, -4);
\draw (1, 2, -4) -- (1, 3, -5);
\draw (1, 2, -4) -- (2, 2, -5);
\draw (2, 1, -4) -- (2, 2, -5);
\draw (2, 1, -4) -- (3, 1, -5);
\draw (2, 2, -3) -- (1, 2, -4);
\draw (2, 2, -3) -- (2, 1, -4);
\draw (2, 2, -3) -- (2, 3, -4);
\draw (2, 2, -3) -- (3, 2, -4);
\draw (2, 3, -4) -- (1, 3, -5);
\draw (2, 3, -4) -- (2, 2, -5);
\draw (3, 2, -4) -- (2, 2, -5);
\draw (3, 2, -4) -- (3, 1, -5);
\draw (3, 3, -3) -- (2, 3, -4);
\draw (3, 3, -3) -- (3, 2, -4);

\draw[gray, dashed] (1, 1, -5) -- (1, 3, -5) -- (3, 3, -5) -- (3, 1, -5) -- cycle;
\draw[gray, dashed] (1, 1, -3) -- (1, 3, -3) -- (3, 3, -3) -- (3, 1, -3) -- cycle;
\draw[gray, dashed] (1, 1, -5) -- (1, 1, -3);
\draw[gray, dashed] (1, 3, -5) -- (1, 3, -3);
\draw[gray, dashed] (3, 3, -5) -- (3, 3, -3);
\draw[gray, dashed] (3, 1, -5) -- (3, 1, -3);
\draw[gray, dashed] (1, 3, -5) -- (3, 1, -5);
\draw[gray, dashed] (1, 1, -3) -- (3, 3, -3);

\end{tikzpicture}
%
\tdplotsetmaincoords{110}{130} 
\begin{tikzpicture}[tdplot_main_coords, scale=1.4]

\filldraw[black] (1, 1, -6) circle (1.5pt) node[anchor=west] {};
\filldraw[black] (1, 2, -7) circle (1.5pt) node[anchor=west] {};
\filldraw[black] (1, 3, -8) circle (1.5pt) node[anchor=west] {};
\filldraw[black] (1, 4, -9) circle (1.5pt) node[anchor=west] {};
\filldraw[black] (2, 1, -7) circle (1.5pt) node[anchor=west] {};
\filldraw[black] (2, 2, -8) circle (1.5pt) node[anchor=west] {};
\filldraw[black] (2, 2, -6) circle (1.5pt) node[anchor=west] {};
\filldraw[black] (2, 3, -9) circle (1.5pt) node[anchor=west] {};
\filldraw[black] (2, 3, -7) circle (1.5pt) node[anchor=west] {};
\filldraw[black] (2, 4, -8) circle (1.5pt) node[anchor=west] {};
\filldraw[black] (3, 1, -8) circle (1.5pt) node[anchor=west] {};
\filldraw[black] (3, 2, -9) circle (1.5pt) node[anchor=west] {};
\filldraw[black] (3, 2, -7) circle (1.5pt) node[anchor=west] {};
\filldraw[black] (3, 3, -8) circle (1.5pt) node[anchor=west] {};
\filldraw[black] (3, 3, -6) circle (1.5pt) node[anchor=west] {};
\filldraw[black] (3, 4, -7) circle (1.5pt) node[anchor=west] {};
\filldraw[black] (4, 1, -9) circle (1.5pt) node[anchor=west] {};
\filldraw[black] (4, 2, -8) circle (1.5pt) node[anchor=west] {};
\filldraw[black] (4, 3, -7) circle (1.5pt) node[anchor=west] {};
\filldraw[black] (4, 4, -6) circle (1.5pt) node[anchor=west] {};

\draw (1, 1, -6) -- (1, 2, -7);
\draw (1, 1, -6) -- (2, 1, -7);
\draw (1, 2, -7) -- (1, 3, -8);
\draw (1, 2, -7) -- (2, 2, -8);
\draw (1, 3, -8) -- (1, 4, -9);
\draw (1, 3, -8) -- (2, 3, -9);
\draw (2, 1, -7) -- (2, 2, -8);
\draw (2, 1, -7) -- (3, 1, -8);
\draw (2, 2, -8) -- (2, 3, -9);
\draw (2, 2, -8) -- (3, 2, -9);
\draw (2, 2, -6) -- (1, 2, -7);
\draw (2, 2, -6) -- (2, 1, -7);
\draw (2, 2, -6) -- (2, 3, -7);
\draw (2, 2, -6) -- (3, 2, -7);
\draw (2, 3, -7) -- (1, 3, -8);
\draw (2, 3, -7) -- (2, 2, -8);
\draw (2, 3, -7) -- (2, 4, -8);
\draw (2, 3, -7) -- (3, 3, -8);
\draw (2, 4, -8) -- (1, 4, -9);
\draw (2, 4, -8) -- (2, 3, -9);
\draw (3, 1, -8) -- (3, 2, -9);
\draw (3, 1, -8) -- (4, 1, -9);
\draw (3, 2, -7) -- (2, 2, -8);
\draw (3, 2, -7) -- (3, 1, -8);
\draw (3, 2, -7) -- (3, 3, -8);
\draw (3, 2, -7) -- (4, 2, -8);
\draw (3, 3, -8) -- (2, 3, -9);
\draw (3, 3, -8) -- (3, 2, -9);
\draw (3, 3, -6) -- (2, 3, -7);
\draw (3, 3, -6) -- (3, 2, -7);
\draw (3, 3, -6) -- (3, 4, -7);
\draw (3, 3, -6) -- (4, 3, -7);
\draw (3, 4, -7) -- (2, 4, -8);
\draw (3, 4, -7) -- (3, 3, -8);
\draw (4, 2, -8) -- (3, 2, -9);
\draw (4, 2, -8) -- (4, 1, -9);
\draw (4, 3, -7) -- (3, 3, -8);
\draw (4, 3, -7) -- (4, 2, -8);
\draw (4, 4, -6) -- (3, 4, -7);
\draw (4, 4, -6) -- (4, 3, -7);

\draw[gray, dashed] (1, 1, -9) -- (1, 4, -9) -- (4, 4, -9) -- (4, 1, -9) -- cycle;
\draw[gray, dashed] (1, 1, -6) -- (1, 4, -6) -- (4, 4, -6) -- (4, 1, -6) -- cycle;
\draw[gray, dashed] (1, 1, -9) -- (1, 1, -6);
\draw[gray, dashed] (1, 4, -9) -- (1, 4, -6);
\draw[gray, dashed] (4, 4, -9) -- (4, 4, -6);
\draw[gray, dashed] (4, 1, -9) -- (4, 1, -6);
\draw[gray, dashed] (1, 4, -9) -- (4, 1, -9);
\draw[gray, dashed] (1, 1, -6) -- (4, 4, -6);

\end{tikzpicture}
\caption{The poset $\tetracoord$ for $n=3,4,5$.}
\label{fig_tetra}
\end{figure}

\begin{prop}\label{JI=tetracoord}
We have
\[\JI \cong \tetracoord\]
as posets via the assignment above.
\end{prop}
\begin{proof}
This is explained in \cite[Section 4.1, 4.2]{kmmBig} (note again that $\JI$ is the set of bigrassmannian elements in $\sbf_n$).
\end{proof}



\section{Comparing subquotients}

Statements in the following section involve taking the sum of certain subquotients in $\domver$.
Here we give a definition to make such statements more precise.

\begin{defn}\label{subsum}
Let $U$ be a set and let $M,N,N_u\in \cO$, for $u\in U$, with $N,N_u\subseteq M$.
\begin{enumerate}
    \item 
We say that \emph{the socle of $M/N$ is contained in the sum of the socles of $M/N_u$} if 
\begin{itemize}
  \item for each $u\in U$, we have $N\subseteq N_u$ (Note that this induces the map
\begin{equation*}\label{sotoso}
    \phi_u:\soc (M/N)\to \soc(M/N_u)
\end{equation*}
for each $u\in U$);
\item\label{3}  we have 
\[\bigcap_{u\in U} \ker \phi_u=0.\]
\end{itemize}
\item
We say that \emph{the socle of $M/N$ contains the sum of the socles of $M/N_u$} if
\begin{itemize}
  \item for each $u\in U$, we have $N\subseteq N_u$ (and the induced map \eqref{sotoso});
\item\label{2} each $\phi_u$ is surjective.
\end{itemize}
\item\label{agrees}
We say that \emph{the socle of $M/N$ agrees with the sum of the socles of $M/N_u$} if it contains and is contained in the socles of $M/N_u$.
\end{enumerate}
\end{defn}

We shall also use the usual (in)equalities when comparing the subquotients of $M$, e.g., write
\begin{equation}
\soc(M/N)  = \sum_{u\in U}\soc(M/N_u) 
\end{equation}
for Definition~\ref{subsum}~\eqref{agrees}.

\section{Submodules and subquotients of the dominant Verma module}\label{s:sub}

Let us consider the poset of (graded) submodules of the dominant Verma module
\[\sub = (\{M\in \cat \ |\ M\subseteq \domver\},\supseteq)\]
and the subposet 
\begin{equation}
\begin{split}
    \shadows
    = \{M\in \sub\ |\ [\soc(\domver / M): L_z] =0 \text{ for $z\not\in \jc$}\}.
\end{split}
\end{equation}
Also consider
\[\tetra =\{M\in\sub\ |\ \soc(\domver / M) \cong L_z \text{ for $z\in\jc$}\} =\{M\in\shadows\ |\ \soc(\domver/ M) \text{ is simple}\}.\]



\begin{prop}\label{tetra=BG}
We have
\begin{equation}\label{tetra to JI}
\tetra=\{\Delta_x\in\sub\ |\ x\in \JI \}\cong (\JI,\leq)\subset \sbf_n,    
\end{equation}
where the isomorphism is given by $\Delta_x\mapsto x$.
\end{prop}
\begin{proof}
The first equality follows from Proposition~\ref{tetraJIA}
and the poset isomorphism follows from \eqref{posetiso}. 
\end{proof}


\begin{prop}\label{max=soc}
Suppose $U\subset \tetra$ consists of mutually incomparable elements in $\tetra$. Then we have 
\[\soc (\domver/\cap_{M\in U} M) = \sum_{M\in U}\soc(\domver / M) \cong \bigoplus_{M\in U}\soc(\domver / M) .\]
\end{prop}
\begin{proof}
Since each $\soc(\domver/ M)$ is simple and is a full isotypic component in (graded) $\domver$, it is enough to prove the first equality.
For that, we need to consider for each $N\in U$ the map $\phi_N:\soc (\domver/\cap_{M\in U} M)\to \soc(\domver/N)$ induced by $\bigcap_{M\in U} M\inj N$. We need to show that we have $\bigcap_{M\in U}\ker\phi_M=0$ and that each $\phi_N$ is surjective.

If $\bigcap_{M\in U}\ker\phi_M\neq 0$ then there is a simple submodule $L\subset \domver / \cap_{M\in U}M$ so that $\phi_N(L)=0$ for all $N\in U$. 
The latter means that the image of $L$ is zero under the quotient $\domver\to\domver/\cap_{M\in U}M$ which contradicts $L\subset \domver / \cap_{M\in U}M$.
This proves $\bigcap_{M\in U}\ker\phi_M = 0$.

To show the surjectivity, note that each $\soc(\domver/N)$ is simple, and thus we need only show all $\phi_N$ are nonzero. 
Suppose not, and take $N\in U$ such that $\phi_N=0$. 
There is a socle component $L\subseteq \soc \domver / \cap_{M\in U}M$ which is contained in all submodules of $\domver$ containing $\soc(\domver / N)$. 
By the above paragraph and Proposition~\ref{multfreeJ}, we have $L=\soc(\domver / N')$ for some $N'\in U$. 
Note that $N' > N$ (i.e., $N'\subset N$) by construction of $L$.
But this implies that $U$ contains comparable elements, which contradicts the assumption on $U$. 
This proves the surjectivity.
\end{proof}

\begin{prop}\label{shadows=intver}
We have
\[\shadows = \{ \bigcap_{x\in U} \Delta_x \in\sub \ |\ U\subseteq W\} = \overline{\Ver}.\]
\end{prop}
\begin{proof}
Let $U\subseteq W$ and consider $\bigcap_{x\in U}\Delta_x$. By \eqref{posetiso}, we have 
$\bigcap_{x\in U}\Delta_x = \bigcap_{x\in U'}\Delta_x$ where $U'$ is the set of maximal elements in $U$. Now $U'$ consists of incomparable elements, so that Proposition~\ref{max=soc} and Proposition~\ref{socinJ} implies $\bigcap_{x\in U'}\Delta_x\in \shadows$.

Conversely, suppose $M\in \sub$ is such that the socle components in $\domver / M$ belong to $\jc$. Under the bijection $\Psi$ from \eqref{Phi}, the socle components corresponds to a set $U\subset \JI$.
The elements in $U$ are incomparable by Proposition~\ref{tetra=BG} since a socle is semisimple.
By Proposition~\ref{max=soc}, we have
\[\soc(\domver / M) = \soc (\domver /\cap_{x\in U} \Delta_x) ,\]
which implies
\[M = \bigcap_{x\in U}\Delta_x\in\overline{\Ver}.\]
\end{proof}

Consider the set $\ideal(\JI)$ of the poset ideals of $(\JI,\leq)$, i.e., the subsets $U\subseteq \JI$ such that $x\leq y$ for $x\in \JI$ and $y\in U$ implies $x\in U$.
Then $\ideal(\JI)$ is a poset under the inclusion.

\begin{prop}\label{shadow to JIideal}
We have
\[\shadows \cong \ideal(\JI)\]
as posets, extending the isomorphism $\tetra \cong \JI$ from \eqref{tetra to JI}.
\end{prop}
\begin{proof}
By Proposition~\ref{tetra=BG}, the poset morphism $\Omega:\shadows \to \ideal(\JI)$ given by 
\[\Omega(M) = \Psi\inv (\{\text{the subquotients in } \domver / M \text{ isomorphic to some $L_z$ with $z\in \jc$}\})\] is well-defined.
Note that the ideal $\Omega(M)$ is (minimally) generated by the subset 
\[\Psi\inv (\{\text{the socle components in } \domver / M\}).\]
By Proposition~\ref{max=soc}, the map $\Omega$ is also an isomorphism. The inverse is 
\[\Omega\inv(U) = \bigcap_{x\in \max(U)} \Delta_x,\]
where $\max (U)$ denotes the set of maximal elements in $U\subseteq \JI$.
\end{proof}

\section{Alternating sign matrices and Verma modules}

The following characterization of $\al_n$ appears to be well-known and is discussed in \cite[Section 3]{manyfacesAl}. 

\begin{prop}\label{tetracoord=al}
We have a lattice isomorphism
\[\ideal(\tetracoord) \cong \al_n\]
which extends the composition of $\tetracoord\cong \JI$ from Proposition~\ref{tetra to JI} and the embedding $\JI\subseteq \sbf_n\inj \al_n$.
\end{prop}
\begin{proof}
 The identification explained in \cite[Section 3]{manyfacesAl} can be made more explicit by placing the tetrahedron in $\mathbb Z^3$ as in Section~\ref{s:tetra}. Doing so, 
 the composition $\tetracoord\cong \JI \subseteq \sbf_n\inj \al_n$ has the image consisting of all join-irreducible elements in the distributive lattice $\al_n$, latter being characterized in \cite[Section 5]{LS} in terms of the monotone triangles (see also \cite[Section 3]{manyfacesAl} for the bijection between the monotone triangles and $\al_n$).
 The claim now follows from the fundamental theorem of distributive lattices.
\end{proof}


All steps needed in our proof of the main theorem are now established.
\begin{theorem}\label{theoremattheend}
We have a lattice isomorphism
\begin{equation}\label{lastiso}
  \al_n  \cong  \overline{\Ver}
\end{equation}
extending the poset isomorphism $\sbf_n\cong \Ver$ given by $w\mapsto \Delta_w$.
\end{theorem}
\begin{proof}
Proposition~\ref{shadows=intver},
Proposition~\ref{shadow to JIideal},
Proposition~\ref{JI=tetracoord}, and
Proposition~\ref{tetracoord=al} provide a poset isomorphism (which implies that $\overline{\Ver}$ is a lattice).
\end{proof}

\begin{remark}
Let us make the isomorphism \eqref{lastiso} more explicit. Given an element $\bigcap_{x\in U}\Delta_{x}$ in $\overline{\Ver}$, the corresponding element $A(U)\in \al_n$ is obtained as follows.
Consider for each $x\in U\subseteq \sbf_n$ the permutation matrix $[x] = (x_{ij})$ and the partial sums 
\[C_{kl}(x) = \sum_{i\leq k, j\leq l}x_{ij}.\]
Then the $n\times n$ matrix given by
\[C_{kl}(U) = \max_{x\in U} C_{kl}(x)\]
is the partial sum of the alternating sign matrix $A(U)$.
\end{remark}

\begin{remark}
Theorem~\ref{theoremattheend} in particular says that $\overline{\Ver}$ is a lattice. 
The join operation on $\overline{\Ver}$ is the intersection of submodules, more or less by construction. The meet does not agree with the sum of submodules. It would be nice to have an interpretation of the meet operation in terms of the Verma modules. 
\end{remark}

\begin{remark}\label{perip}
In \cite[Theorem 7]{ChM}, certain Verma modules over the periplectic Lie superalgebra $\mathfrak{pe}(n)$ are shown to have a homological property similar to that of $\mathfrak{sl}_n$.
Namely the Verma modules $\Delta_{\mathfrak{pe}}(w.0)$  over $\mathfrak{pe}(n)$ of highest weight $w. 0$, where $w\in \sbf_n$, are submodules of 
$\Delta_{\mathfrak{pe}}(0)$ in a canonical way respecting the Bruhat order on $w\in \sbf_n$ (see \cite[Proposition 3 (1), Theorem 2 (2)]{ChM}).
Moreover, Propositions~\ref{socinJ},~ \ref{multfreeJ},~\ref{tetraJIA} hold for this setting (see \cite[Theorem~7]{ChM}). 
Thus the results in Section~\ref{s:sub} and the proof of Theorem~\ref{theoremattheend} are valid for $\mathfrak{pe}(n)$.
\end{remark}

\bibliographystyle{plain}
\bibliography{bib}

\end{document}